\documentclass[preprint]{elsarticle}

\usepackage{xypic}
\usepackage{eucal}

\usepackage{amsmath}
\usepackage{amstext}
\usepackage{amssymb}
\usepackage{amsthm}
\usepackage{amscd}

\usepackage{lscape}
\usepackage{longtable}

\usepackage{epsfig}
\input txdtools
\let\et=\etexdraw
\def\etexdraw{\drawbb\et}

\theoremstyle{plain}
\newtheorem{thm}{Theorem}[section]
\newtheorem{thm*}{Theorem}
\newtheorem{lem}[thm]{Lemma}
\newtheorem{prop}[thm]{Proposition}
\newtheorem{prop*}[thm*]{Proposition}
\newtheorem{cor}[thm]{Corollary}

\theoremstyle{definition}
\newtheorem{defi}[thm]{Definition}

\newtheorem{ex}[thm]{Example}

\theoremstyle{remark}

\DeclareMathOperator{\Coker}{Coker}
\DeclareMathOperator{\Image}{Im}

\DeclareMathOperator{\Hom}{Hom}

\DeclareMathOperator{\Ann}{ann}
\DeclareMathOperator{\Nil}{Nil}

\DeclareMathOperator{\EEop}{E}
\newcommand{\EE}[1]{\EEop\left( #1 \right)}

\begin{document}

\title{Annihilators of Artinian modules compatible with a Frobenius map}

\author{Mordechai Katzman\fnref{fn1}}
\ead{M.Katzman@sheffield.ac.uk}
\address{Department of Pure Mathematics,
University of Sheffield, Hicks Building, Sheffield S3 7RH, United Kingdom}

\author{Wenliang Zhang\fnref{fn2}}
\ead{wzhang15@unl.edu}
\address{Department of Mathematics, University of Nebraska, 203 Avery Hall, Lincoln, NE 68588, USA}


\fntext[fn1]{The first author thankfully acknowledges support from EPSRC grant EP/I031405/1.}

\fntext[fn2]{The second author is supported in part by NSF Grant DMS \#1068946.}



\begin{abstract}
In this paper we consider Artinian modules over power series rings endowed with a Frobenius map.
We describe a method for finding the set of all prime annihilators of submodules which are preserved by the given Frobenius map and on which the
Frobenius map is not nilpotent.
This extends the algorithm by Karl Schwede and the first author, which solved this problem for submodules of the injective hull of the residue field.

The Matlis dual of this problem asks for the radical annihilators of quotients of free modules by submodules preserved by a given Frobenius near-splitting, and the same method solves this dual problem in the $F$-finite case.
\end{abstract}

\begin{keyword}
Frobenius map \sep Frobenius splitting
\end{keyword}
\maketitle


\section{Introduction}\label{Section: Introduction}

This paper describes an algorithm for finding the annihilators of submodules of Artinian modules which are preserved by a given Frobenius map.

Throughout this paper  $R$ will denote a ring of formal power series  over a field $\mathbb{K}$ of prime characteristic $p$, $\mathfrak{m}$ will denote its maximal ideal,
and $E=E_R(R/\mathfrak{m})$ will denote the injective
hull of its residue field. The Frobenius map sending $r\in R$ to its $p$th power will be denoted $f$, and $f^e$ will be its $e$th iteration.

Given any $R$-module $M$ and $e\geq 0$, we may endow $M$ with a new $R$-module structure given by $r \cdot m = r^{p^e} m$ for all
$r\in R$ and $m\in M$ and we denote this new module $F_*^e M$. An \emph{$e$th Frobenius map} on $M$ is an element of
$\Hom_R (M, F^e_* M)$, or, equivalently, an additive map $\phi: M \rightarrow M$ such that $\phi(r m)= r^{p^e} \phi(m)$ for all $r\in R$ and $m\in M$.
Given such a Frobenius map
$\phi \in \Hom_R (M, F^e_* M)$ we call an $R$-submodule $N\subseteq M$ $\phi$-compatible if $\phi(N)\subseteq F^e_* N$.
When discussing the case $e=1$, we shall  drop the $e$ from the notation above.

The aim of this paper is to find the set of radical annihilators of all $\phi$-compatible submodules of a given Artinian $R$-module, or, equivalently
(cf.~Proposition \ref{Proposition: correspondence between Frobenius maps and matrices} below),
given a $\phi \in \Hom_R (E^\alpha, F_* E^\alpha)$ for some positive integer $\alpha$,
to find all radical annihilators of $R$-submodules $N\subseteq E^\alpha$ which satisfy $\phi(N) \subseteq F_* N$.
We shall accomplish this under the assumption that this $\phi$ restricts to a non-zero map on $N$:
in this case the set of radical annihilators is shown to be finite and given by the intersection of all prime ideals in it
(cf.~\cite[Corollary 3.11]{Sharp} and \cite[Section 3]{Enescu-Hochster}.)

This extends the results in \cite{Katzman-Schwede} which describes an algorithm for producing such sets of annihilators when $\alpha=1$.
We shall first describe this algorithm from a more algebraic point of view than in \cite{Katzman-Schwede} and comment on why it cannot be directly
extended for $\alpha>1$.
We shall then give a description of Frobenius maps on $E^\alpha$ in terms of certain matrices and finally we will produce an algorithm which works recursively on
$\alpha$, in which the case $\alpha=1$ treated in \cite{Katzman-Schwede} provides the foundation.

This paper is organized as follows.
Section \ref{Section: Frobenius maps of $E^n$ and their stable submodules} introduces the notion of Frobenius maps and
studies Frobenius maps of Artinian modules using the properties of a version of Matlis duality which keeps track of the
Frobenius maps: these are the functors $\Delta^e$ and $\Psi^e$ described there.

In section \ref{Extending the star-closure} we generalize
two operations which were originally introduced in the context of Frobenius splittings and Frobenius maps in
the injective hulls of residue fields, namely, the $I_e(-)$ operation (denoted $^{[1/p^e]}$ by some authors) and the $\star$-closure.
These are extended from operations on ideals to operations on submodules of free modules, and some of their properties are studied here, e.g., their behaviour under localization.

Section \ref{Section: the case alpha=1} reviews the algorithm in \cite{Katzman-Schwede} for finding prime annihilators
of submodules of the injective hull of the residue field stable under a given Frobenius map,
and presents a proof for its main ingredient in algebraic language.

The main section of this paper, section \ref{Section: the case alpha=1} generalizes the Katzman-Schwede algorithm to deal with prime
annihilators of general Artinian modules endowed with a Frobenius map.

The main result, Theorem \ref{Theorem: Main Theorem}, yields an algorithm which is described in detail in section
\ref{Section: The algorithm in action and two calculations}. We also carry out two calculations following the algorithm to illustrate its use.

Finally, section \ref{Section: Connections with Frobenius near splittings} translates the previous results into the language of Frobenius
near-splittings of  free modules in the case where we work over an $F$-finite ring: in this setup Frobenius maps and near-splittings are dual notions.

\section{Frobenius maps of Artinian modules and their stable submodules}
\label{Section: Frobenius maps of $E^n$ and their stable submodules}

In this section we describe all Frobenius maps on Artinian $R$-modules.
We may  think of $e$th Frobenius maps as left-module structures over the following skew-commutative rings $R[\Theta; f^e]$:
as an $R$-module it is the free module $\displaystyle \oplus_{i=0}^\infty R \Theta^{i}$
and we extend the rule $\Theta r=r^{p^e}\Theta$ for all $r\in R$ to a (non-commutative!) multiplicative structure on $R[\Theta; f^e]$.
Now given an $e$th Frobenius map $\phi$ on an $R$-module $M$, we can turn it into a left
$R[\Theta; f^e]$-module by extending the rule $\Theta m=\phi(m)$ for all $m\in M$.
The fact that this gives $M$ the structure of a left $R[\Theta; f^e]$-module is simply because
for all $r\in R$ and $m\in M$,
$$\Theta (r m) = \phi(rm)=r^{p^e} \phi(m)= r^{p^e} \Theta m = (\Theta r) m .$$
Conversely, if $M$ is a left $R[\Theta; f^e]$-module, then $\Theta: M \rightarrow M$ is an $e$th Frobenius map.

Recall the definition of the \emph{$e$th Frobenius functor}:
the tensoring $(-) \rightarrow F_*^e R \otimes_R (-)$
defines a functor from the category of $R$-modules to the category of $F_*^e R$-modules.
We may now identify the rings $R$ and $F_*^e R$ and to obtain  $e$th Frobenius functor
$F_R^e(-)$
from the category of $R$-modules to itself.

Following \cite{Katzman1} we shall refer to
the category of Artinian $R[\Theta;f^e]$-modules $\mathcal{C_e}$ and
the category $\mathcal{D_e}$  of $R$-linear maps $M \rightarrow F^e_R(M)$ where $M$ is
a finitely generated $R$-module,  and where a morphism between $M\xrightarrow[]{a} F^e_R(M)$ and
$N\xrightarrow[]{b} F^e_R(N)$ is a commutative diagram of $R$-linear maps
\begin{equation*}
\xymatrix{
M \ar@{>}[d]^{a} \ar@{>}[r]^{\mu} & N \ar@{>}[d]^{b}\\
F^e_R(M) \ar@{>}[r]^{F^e_R(\mu)} & F_R(N)\\
} .
\end{equation*}

We also refer to the mutually inverse functors
$\Delta^e: \mathcal{C_e} \rightarrow \mathcal{D}$ and
$\Psi^e: \mathcal{D_e} \rightarrow \mathcal{C}$
also introduced in \cite{Katzman1}.
These are extensions of Matlis duality functors $(-)^\vee=\Hom_R( -, E)$ which, additionally, keep track of Frobenius actions and are defined as follows.
Given an $R[\Theta; f^e]$-module $M$,  $\Delta^e(M)$ is defined (functorially) as
Matlis dual of the $R$-linear map $F^e_* R \otimes_R M \rightarrow M$ given by $r \otimes m \mapsto r \Theta m$
where $(F^e_* R \otimes_R M)^\vee$ is identified with $F^e_R(M^\vee)$ (cf.~\cite[Lemma 4.1]{Lyubeznik}.)
Given an $R$-linear map in $\mathcal{D}_e$, one can reverse the steps of the construction of $\Delta^e$ and obtain functorially
an Artinian module with a Frobenius map defined on it.

As before, we will suppress $e$ from the notation when $e=1$.

\bigskip
Given an Artinian $R$-module $M$ we can embed $M$ in $E^\alpha$ for some $\alpha$ and extend this inclusion to an exact sequence
$$ 0\rightarrow M\rightarrow E^\alpha \xrightarrow[]{A^t} E^\beta \rightarrow \dots$$
where
$A^t\in \Hom_R(E^\alpha, E^\beta)\cong \Hom_R(R^\alpha, R^\beta)$ is
a $\beta\times \alpha$ matrix with entries in $R$.
Proposition \ref{Proposition: correspondence between Frobenius maps and matrices} below shows that the
Frobenius maps on $M$ are restrictions of Frobenius maps on $E^\alpha$ and those can be described in terms of the following
canonical Frobenius map $T: E^\alpha \rightarrow E^\alpha$.

Since $R$ is regular local, $E$ is isomorphic to the module of inverse polynomials $\mathbb{K}[\![ x_1^-, \dots, x_d^- ]\!]$
where $x_1, \dots, x_d$ are minimal generators of the maximal ideal of $R$ (cf.~\cite[\S 12.4]{Brodmann-Sharp}.)
Thus $E$ has a natural $R[T; f]$-module structure
additively extending $T (\lambda x_1^{-\alpha_1} \dots x_1^{-\alpha_d})= \lambda^p x_1^{-p \alpha_1} \dots x_1^{-p \alpha_d}$ for
$\lambda\in \mathbb{K}$ and $\alpha_1, \dots, \alpha_d>0$.
We can further extend this to a natural $R[T; f]$-module structure on $E^\alpha$ given by
$$T \left( \begin{array}{c} a_1 \\ \vdots \\ a_\alpha  \end{array} \right) =  \left( \begin{array}{c} T a_1 \\ \vdots \\ T a_\alpha \end{array} \right) .$$


\begin{prop}\label{Proposition: correspondence between Frobenius maps and matrices}
Let $M=\ker A^t$ be an  Artinian  $R$-module where $A$ is a $\alpha \times \beta$ matrix with entries in $R$.
Let $e\geq 1$ and let $\mathbf{B}$ be the set of $\alpha \times \alpha$ matrices which satisfy
$\Image B A \subseteq \Image A^{[p^e]}$.
For a given $e$th Frobenius map on $M$,
$\Delta^e(M)\in \Hom_R(\Coker A, \Coker A^{[p^e]})$ and is given by multiplication by a matrix  $B$ in $\mathbf{B}$ and, conversely,
any such $B$ defines an $R[\Theta; f^e]$-module structure on $M$ which is given by the restriction to $M$ of the Frobenius map
$\phi: E^\alpha\rightarrow E^\alpha$
defined by $\phi(v)=B^t T^e(v)$ where $T$ is the natural Frobenius map on $E^\alpha$.
\end{prop}
\begin{proof}
Matlis duality gives an exact sequence
$ R^\beta \xrightarrow[]{A} R^\alpha \rightarrow M^\vee \rightarrow 0$ hence
$$\Delta^e(M)\in \Hom_R( M^\vee, F^e_R(M^\vee))=\Hom_R(\Coker A, \Coker A^{[p^e]}).$$

Let $\Delta^e(M)$ be the map $\phi:\Coker A \rightarrow \Coker A^{[p^e]}$.

In view of Theorem 3.1 in \cite{Katzman1} we only need to show that
any such $R$-linear map is given by multiplication by a matrix $B$ in $\mathbf{B}$, and that any such $B$
defines an element in $\Delta^e(M)$.

The freeness of $R^\alpha$ enables us to lift the map $\phi:  \Coker A \rightarrow \Coker A^{[p^e]}$ to a map
$\phi^\prime: R^\alpha \rightarrow R^\alpha$ given by multiplication by some $\alpha \times \alpha$ matrix $B$ in $\mathbf{B}$.
Conversely, any such matrix $B$ defines a map $\phi: \Coker A \rightarrow \Coker A^{[p^e]}$, and $\Psi^e(\phi)$ is a Frobenius map on
$M$ as described in the statement of the proposition.
\end{proof}

In the rest of the paper we shall consider Frobenius actions $\Theta=U T$ on $E^\alpha$ and $R[\Theta^e; f^e]$
submodules $M\subseteq E^\alpha$.
The proposition above shows that for any such $M$ there is a $V\subseteq R^{\alpha}$ such that
$M=\Ann_{E^\alpha} V^t:= \{ z\in E^\alpha \,|\, V^t z=0\}$ and $U V \subseteq V^{[p^e]}$.
This will be henceforth used extensively and implicitly.
For simplicity we adopt the following notation:
given any $V\subseteq R^{\alpha}$ we define $\EE{V}=\Ann_{E^\alpha} V^t$.

\section{Extending the $\star$-closure}
\label{Extending the star-closure}

The purpose of this section is to extend the $\star$-closure operation as first defined in section 5 of \cite{Katzman1}.

\begin{defi}
Let $e\geq 0$.
\begin{enumerate}
  \item[(a)] Given any matrix (or vector) $A$ with entries in $R$, we define $A^{[p^e]}$ to be the matrix obtained from $A$ by raising its
entries to the $p^e$th power.

  \item[(b)] Given any submodule $K\subseteq R^\alpha$, we define $K^{[p^e]}$ to be the $R$-submodule of $R^\alpha$ generated by
$\{ v^{[p^e]} \,|\, v\in K \}$.

\end{enumerate}
\end{defi}

The theorem below extends the $I_e(-)$ operation defined on ideals in \cite[Section 5]{Katzman1}
and in \cite[Definition 2.2]{Blickle-Mustata-Smith} (where it is denoted $(-)^{[1/p^e]}$) to submodules of free $R$-modules.

\begin{thm}
\label{qth root with respect to U}
Let $e\geq 1$.
\begin{enumerate}
  \item[(a)] Given a submodule $K\subseteq R^\alpha$ there exists a minimal submodule $L \subseteq R^\alpha$ for which
  $K\subseteq L^{[p^e]}$. We denote this minimal submodule $I_e (K)$.
  \item[(b)] Let  $U$ be a $\alpha\times \alpha$ matrix with entries in $R$ and let $V\subseteq R^\alpha$.
  The set of all submodules $W \subseteq R^\alpha$ which contain $V$ and which satisfy $U W \subseteq W^{[p^e]}$ has a unique
  minimal element.
\end{enumerate}
\end{thm}
\begin{proof}
Let $L$ be the intersection of all submodules $M\subseteq R^\alpha$  for which $K\subseteq M^{[p^e]}$.
Proposition 5.3 in \cite{Katzman1} implies that $K\subseteq L^{[p^e]}$ and clearly, $L$ is minimal with this property.

To prove (b) we carry out a construction similar to that in \cite[section 5]{Katzman1}.
Define inductively $V_0=V$ and
$V_{i+1}=I_1(U V_i) + V_i$ for all $i\geq 0$.
The sequence $\{ V_i \}_{i\geq 0}$ must stabilize to some submodule $W=V_j \subseteq R^\alpha$.
Since $W=I_1 (U W) + W$, $I_1 (U W) \subseteq W$ and $U W \subseteq W^{[p]}$.

Let $Z$ be any submodule of $R^\alpha$ containing $V$ for which $U Z \subseteq Z^{[p]}$. We show by induction on $i$ that
$V_i \subseteq Z$ for all $i\geq 0$. Clearly, $V_0=V\subseteq Z$, and if for some $i\geq 0$, $V_i\subseteq Z$ then
$U V_i \subseteq U Z \subseteq Z^{[p]}$ hence $I_1(U V_i) \subseteq Z$ and $V_{i+1} \subseteq Z$. This shows that $W\subseteq Z$.
\end{proof}

\begin{defi}
With notation as in Theorem \ref{qth root with respect to U}, we call the unique minimal submodule in \ref{qth root with respect to U}(b) the \emph{star closure of $V$ with respect to $U$}
and denote it
$V^{\star U}$.
\end{defi}

The effective calculation of the $\star$-closure boils down to the calculation of $I_e$, and this is
a straightforward generalization of the calculation of $I_e$ for ideals. To do so, we first note that if
$R$ is a free $R^p$-module with free basis $\mathcal{B}$ (e.g., when $\dim_{\mathbb{K}^p} \mathbb{K}< \infty$), then
every element $v\in R^\alpha$ can be expressed uniquely in the form $v=\sum_{b\in \mathcal{B}} u_{b}^{[p^e]} b$
where $u_{b}\in R^\alpha$ for all $b\in \mathcal{B}$.

\begin{prop} \label{Proposition: Computing Ie}
Let $e\geq 1$.
\begin{enumerate}
  \item [(a)] For any submodules $V_1, \dots, V_\ell\subseteq R^n$, $I_e(V_1 + \dots + V_\ell)=I_e(V_1)  + \dots + I_e(V_\ell)$.
  \item [(b)] Assume that $R$ is a free $R^p$-module with free basis $\mathcal{B}$ (e.g., when $\dim_{\mathbb{K}^p} \mathbb{K}< \infty$).
  Let $v\in R^\alpha$ and let
  $$v=\sum_{b\in \mathcal{B}} u_{b}^{[p^e]} b $$
  be the unique expression for $v$ where $u_{b}\in R^\alpha$ for all
  $b\in \mathcal{B}$. Then  $I_e(R v)$ is the submodule $W$ of $R^\alpha$ generated by $\{ u_b  \,|\, b\in \mathcal{B} \}$.
  \end{enumerate}
\end{prop}

\begin{proof}
The proof of this proposition is a straightforward modification of the proofs of propositions 5.2 and 5.6 in \cite{Katzman1}
and Lemma 2.4 in \cite{Blickle-Mustata-Smith}.

Clearly, $I_e(V_1 + \dots + V_\ell)\supseteq I_e(V_i)$ for all $1\leq i\leq \ell$, hence $I_e(V_1 + \dots + V_\ell)\supseteq I_e(V_1)  + \dots + I_e(V_\ell)$.
On the other hand
$$(I_e(V_1)  + \dots + I_e(V_\ell))^{[p^e]} = I_e(V_1)^{[p^e]}  + \dots + I_e(V_\ell)^{[p^e]} \supseteq V_1+\dots+ V_\ell $$
and the minimality of $I_e(V_1 + \dots + V_\ell)$ implies that
$I_e(V_1 + \dots + V_\ell)\subseteq I_e(V_1)  + \dots + I_e(V_\ell)$ and (a) follows.

Clearly $v\in W^{[p^e]}$, and so $I_e(R v) \subseteq W$. On the other hand, let $W$ be a submodule of $R^\alpha$ such that $v\in W^{[p^e]}$.
Write $v = \sum_{i=1}^s r_i w_i^{[p^e]}$ for $r_i\in R$ and $w_i\in W$ for all $1\leq i \leq s$, and for each such $i$
write $r_i=\sum_{b\in \mathcal{B}} r_{b i}^{p^e} b$ where $r_{b i}\in R$ for all $b\in \mathcal{B}$.
Now
$$ \sum_{b\in \mathcal{B}} u_{b}^{[p^e]} b = v = \sum_{b\in \mathcal{B}} \left(\sum_{i=1}^s r_{b i}^{p^e} w_i^{[p^e]} \right) b $$
and since these are direct sums, we compare coefficients and obtain
$u_{b}^{[p^e]} = \left(\sum_{i=1}^s r_{b i}^{p^e} w_i^{[p^e]} \right)$ for all $b\in \mathcal{B}$
and so
$u_{b} = \left(\sum_{i=1}^s r_{b i} w_i \right)$ for all $b\in \mathcal{B}$
hence
$u_{b} \in W$ for all $b\in \mathcal{B}$.
\end{proof}

\bigskip
\begin{lem}[cf.~\cite{Murru}]\label{Lemma: Ie and localization}
Let $\mathcal{S}\subset R$ be a multiplicative set, and let $W\subseteq R^\alpha$.
For all $e\geq 1$, $I_e(\mathcal{S}^{-1} W)$ exists and equals $\mathcal{S}^{-1} I_e(W)$.
\end{lem}

\begin{proof}
We first note that $I_e(\mathcal{S}^{-1} W \cap R^\alpha)^{[p^e]} \supseteq \mathcal{S}^{-1} W \cap R^\alpha$ hence
$\mathcal{S}^{-1} I_e(\mathcal{S}^{-1}W  \cap R^\alpha)^{[p^e]} \supseteq \mathcal{S}^{-1}(\mathcal{S}^{-1}W \cap R^\alpha) = \mathcal{S}^{-1}W$.

Let $W_1\subseteq R^\alpha$ be another module for which $\mathcal{S}^{-1} W_1 ^{[p^e]} \supseteq \mathcal{S}^{-1} W$; we have
$$\left(\mathcal{S}^{-1} W_1  \cap R^\alpha\right)^{[p^e]} =\mathcal{S}^{-1} W_1 ^{[p^e]} \cap R^\alpha  \supseteq
\mathcal{S}^{-1} W \cap R^\alpha$$
hence
$I_e(\mathcal{S}^{-1} W \cap R^\alpha) \subseteq \mathcal{S}^{-1} W_1  \cap R^\alpha$ and
$\mathcal{S}^{-1} I_e(\mathcal{S}^{-1} W \cap R^\alpha) \subseteq \mathcal{S}^{-1}\left(\mathcal{S}^{-1} W_1  \cap R^\alpha\right)=\mathcal{S}^{-1} W_1$.

We can now conclude that $I_e(\mathcal{S}^{-1} W)$ exists and equals $\mathcal{S}^{-1} I_e(\mathcal{S}^{-1} W \cap R^\alpha)$.

We now have $I_e(\mathcal{S}^{-1} W)=\mathcal{S}^{-1} I_e(\mathcal{S}^{-1} W \cap R^\alpha) \supseteq \mathcal{S}^{-1} I_e(W)$, and we finish the proof by
showing that  $I_e(\mathcal{S}^{-1} W) \subseteq \mathcal{S}^{-1} I_e(W)$. This last inclusion is equivalent to
$\mathcal{S}^{-1} W \subseteq \left(\mathcal{S}^{-1} I_e(W) \right)^{[p^e]}$
and this follows from the fact that $W\subseteq I_e(W)^{[p^e]}$.
\end{proof}

The existence of the $I_e(-)$ operation in localizations of $R^\alpha$ allows us to define ${}^{\star U}$ operations on submodules
of these localizations in an identical way to its definition for submodules of $R^\alpha$.
We shall use these later in Section \ref{Section: the case alpha>1}.

\bigskip
Throughout the rest of this section we fix a Frobenius map $\Theta=U^t T : E^\alpha \rightarrow  E^\alpha$ where
$U$ is an $\alpha \times \alpha$ matrix with entries in $R$. Recall that, given any $V\subseteq R^{\alpha}$ we use $\EE{V}$ to denote $\Ann_{E^\alpha} V^t$. We will collect some properties of $\EE{V}$ which will be used later in Section \ref{Section: the case alpha>1}.

\begin{lem}[cf.~Theorem 4.7 in \cite{Katzman1}]\label{Lemma: Zero action submodule}
The $R[\Theta; f]$-submodule
$Z=\{ a\in E^\alpha \,|\, \Theta^e a =0 \}$ is
given by
$$\EE{ I_e(\Image U^{[p^{e-1}]} U^{[p^{e-2}]} \dots U )} .$$
\end{lem}
\begin{proof}
Write $Z=\EE{ W}$ for some $R$-submodule $W\subseteq R^\alpha$.
We may view $E^\alpha$ and $Z$ as $R[\Theta^e; f^e]$-modules and an application
$\Delta^e$ to the inclusion $Z\subseteq E^{\alpha}$ gives a commutative diagram with exact rows
\begin{equation*}\label{CD1}
\xymatrix{
R^\alpha \ar@{>}[d]_{U^{[p^{e-1}]} U^{[p^{e-2}]} \cdots U } \ar@{>}[r]^{}   & R^\alpha/W \ar@{>}[d]^{U^{[p^{e-1}]} U^{[p^{e-2}]} \cdots U } \ar@{>}[r]^{} &  0\\
R^\alpha \ar@{>}[r]^{}                  & R^\alpha/W^{[p^e]} \ar@{>}[r]^{}          &  0\\
}
\end{equation*}
where the right-most vertical map is zero.
We deduce that $W$ is the smallest submodule of $R^\alpha$ for which the rightmost vertical map is zero, i.e.,
$W=I_e( \Image U^{[p^{e-1}]} U^{[p^{e-2}]} \cdots U )$.

\end{proof}

\begin{lem}[cf.~Theorem 4.8 in \cite{Katzman3}]\label{Lemma: I1(uK)}
Let $K\subseteq R^\alpha$ and assume that $\EE{K}$ is an $R[\Theta; f]$-module.
The $R[\Theta; f]$-module
$M=\{ z\in E^\alpha \,|\, \Theta z \in \EE{K} \}$ is $\EE{I_1(U K)}$.
\end{lem}

\begin{proof}
Since $\EE{K}$ is an $R[\Theta; f]$-module, $\EE{K}\subseteq M$.
Write $M=\EE{V}$ for some $V\subseteq R^\alpha$ and apply $\Delta^1$ to the short exact sequence
$0 \rightarrow  \EE{K}  \rightarrow \EE{V} \rightarrow  \EE{V}/  \EE{K} \rightarrow 0$
of $R[\Theta; f]$-modules to obtain the following commutative diagram with exact rows
\begin{equation*}
\xymatrix{
0  \ar@{>}[r]^{} & K/V  \ar@{>}[r]^{}  \ar@{>}[d]^{U}   &  R^\alpha/V  \ar@{>}[r]^{}  \ar@{>}[d]^{U}  &  R^\alpha/K  \ar@{>}[r]^{}  \ar@{>}[d]^{U}  & 0 \\
0  \ar@{>}[r]^{} & K^{[p]}/V^{[p]}  \ar@{>}[r]^{}       &  R^\alpha/V^{[p]}  \ar@{>}[r]^{}            &  R^\alpha/K^{[p]}  \ar@{>}[r]^{}            & 0 \\
}
\end{equation*}
and $V$ is the smallest submodule of $R^\alpha$ on which the leftmost vertical map vanishes, i.e., $V=I_1(U K)$.
\end{proof}

\begin{lem}\label{Lemma: associated of special is special}
Let $\EE{W}$ be a $R[\Theta; f]$-submodule, where $W\subseteq R^\alpha$.
Write $J=(0 :_R \EE{W})=(0 :_R R^\alpha/W)$ and let $Q$ be an associated prime of $J$.
There exists an $R$-submodule $\widehat{W}\subseteq R^\alpha$ such that $\EE{\widehat{W}}$ is a $R[\Theta; f]$-submodule
and
$(0 :_R \EE{\widehat{W}})=(0 :_R R^\alpha/\widehat{W})=Q$.
\end{lem}

\begin{proof}
Let $q_1 \cap \dots \cap q_s$ be a minimal primary decomposition of $J$ with $Q=\sqrt{q_1}$.
Pick an $a\in R$ for which $(J:a)=Q$.
and write $\widehat{W}=(W :_{R^\alpha} a) :=\{ v\in R^\alpha \,|\, a v \in W\}$.
It is straightforward to verify that
$(0 :_R R^\alpha/\widehat{W})=(J : a ) = Q$ and since $U W \subseteq W^{[p]}$, we have
$$a^p U \widehat{W} \subseteq a^{p-1} U W \subseteq a^{p-1} W^{[p]} \subseteq   W^{[p]}$$
and so $U \widehat{W} \subseteq (W^{[p]} :_{R^\alpha} a^p)= (W :_{R^\alpha} a)^{[p]}$. 
\end{proof}

Next, we want to introduce the following terminology.

\begin{defi}
Let $\Theta = U^t T : E^\alpha \rightarrow E^\alpha$ (where $U$ is a $\alpha \times \alpha$ matrix with entries in $R$) be a Frobenius map.
We shall call an ideal \emph{$\Theta$-special} (or just \emph{special} if $\Theta$ is understood) if it is an annihilator
of an $R[\Theta; f]$-submodule of $E^\alpha$.
Equivalently, an ideal is $\Theta$-special if it is the annihilator of $R^\alpha/W$ where $UW \subseteq W^{[p]}$.

A $\Theta$-special prime ideal shall be referred to as being \emph{$\Theta$-special prime}.
\end{defi}

A basic fact concerning special primes is the following.

\begin{lem}
Let $P\subset R$ be a special prime and $V=\left( P R^\alpha \right)^{\star U} \subseteq R^\alpha$.
Then $\EE{V}$ is the largest $R[\Theta; f]$-module
whose annihilator is $P$.
\end{lem}
\begin{proof}
The construction of the $\star$-closure guarantees that $\EE{ V}$ is an $R[\Theta; f]$-module, and this is clearly annihilated by $P$.
If $\EE{W}$ is another $R[\Theta; f]$-module annihilated by $P$ then $P R^\alpha \subseteq W$ and
$\left( P R^\alpha \right)^{\star U} \subseteq  W^{\star U}= W$  and hence $\EE{ V} \supseteq \EE{W}$ and the annihilator of
both is $P$.
\end{proof}


\section{The case $\alpha=1$}
\label{Section: the case alpha=1}

In this section we describe an algorithm for finding all submodules of $\EE{ P} \subset E$ which are preserved by a given Frobenius map $\Theta=u T$ ($u\in R$),
under the assumptions that $P\subset R$ is prime and
that the restriction of $\Theta : E \rightarrow E$ to  $\EE{P}$ is not the zero map. This algorithm is essentially the one described in \cite{Katzman-Schwede}, however,
we present it here in terms
of $R[\Theta; f]$-submodules of $E$ rather than in terms of Frobenius splittings and we do so in more algebraic language.

Fix $u\in R$ and $\Theta=uT$ throughout the rest of this section.

\begin{thm}[cf.~section 4 in \cite{Katzman-Schwede}]\label{Theorem: the Utah Proof}
Let $P\subset Q$ be prime $\Theta$-special ideals, write $S=R/P$.
Let $J\subseteq R$ be an ideal whose image in $S$ defines its singular locus.
\begin{enumerate}
  \item[(a)] If $(P^{[p]}:P) Q \subseteq Q^{[p]}$ then $J\subseteq Q$.
  \item[(b)] If $(P^{[p]}:P) Q \nsubseteq Q^{[p]}$ then $( uR + P^{[p]} : (P^{[p]}:P))  \subseteq Q$.
  \item[(c)] Assume further that $R$ is $F$-finite. If the restriction $\Theta$ to $\EE{P}$ is not the zero map,
  then $( uR + P^{[p]} : (P^{[p]}:P)) \supsetneq P$.
\end{enumerate}
\end{thm}

\begin{proof}
Write $E_{R_Q}=E_{R_Q}(R_Q/QR_Q)$.
Note that $R_Q$ is regular; let $\widetilde{T}$ be the natural Frobenius on $E_{R_Q}$.

Write $\widetilde{E}=E_{S_Q}(S_Q/QS_Q)=\Ann_{E_{R_Q}} P R_Q$ and note that
the Frobenius maps on $\widetilde{E}$ are  given by $(PR_Q^{[p]}:PR_Q) \widetilde{T}$
and that the Frobenius maps on $\Ann_{E_{R_Q}} Q R_Q \subset \Ann_{E_{R_Q}} P R_Q=\widetilde{E}$
are  given by $(QR_Q^{[p]}:QR_Q) \widetilde{T}$ (cf.~Proposition 4.1 in \cite{Katzman1}).

If (a), then $(PR_Q^{[p]}:PR_Q) \subseteq (QR_Q^{[p]}: QR_Q)$, i.e.,
 $\Ann_{E_{R_Q}} Q R_Q$ is an $S[\theta]$-submodule of $\widetilde{E}$ for all Frobenius maps $\theta$ on $\widetilde{E}$.
Now if $J\nsubseteq Q$, $S_Q$ is regular and $\widetilde{E}$ is a simple $S[\tau]$-module  where
$\tau:  \widetilde{E} \rightarrow\widetilde{E}$ is the natural Frobenius map,
hence $Q=P$ or $Q=R$, a contradiction.

If (b), pick any $c\in ( uR  + P^{[p]} : (P^{[p]}:P))$.
We have
$$c (P^{[p]}:P) Q \subseteq (uR  + P^{[p]}) Q \subseteq Q^{[p]}$$
and since  $(P^{[p]}:P) Q \nsubseteq Q^{[p]}$ we conclude
$c$ is a zero-divisor on $R/Q^{[p]}$ and $c\in Q$.

To prove (c) we follow \cite{Fedder} and identify the $S$-module $(P^{[p]} : P)/ P^{[p]}$ with $\Hom_{S} (F_* S, S)$
and $u$ with a non-zero $\psi\in \Hom_{S} (F_* S, S)$. We define $C$ to be the $S$-submodule of $\Hom_{S} (F_* S, S)$ generated
by $\psi$. Now $\Hom_{S} (F_* S, S)$ is a rank-one $F_* S$-module (cf.~\cite[Lemma 1.6]{Fedder}) and hence there exists a non-zero $c\in S$
which multiplies $\Hom_{S} (F_* S, S)$ into $C$, and hence $c$ multiplies $(P^{[p]} : P)$ into $uR + P^{[p]}$.
\end{proof}

To turn this theorem into an algorithm, one would start with a given special prime $P$ and find all
special primes $Q\supsetneq P$ for which there is no special prime strictly between $P$ and $Q$.
We shall henceforth refer to such special prime $Q$ as \emph{minimally containing $P$}.

\begin{cor}\label{Corollary: finitely many minimally minimally containing primes}
\begin{enumerate}
  \item [(a)] Any prime containing $I_1(uR)$ is a special prime.
  \item [(b)] Let $P$ be a special prime which does not contain $I_1(uR)$. The set of special primes minimally containing
$P$ is finite.
  \item [(c)] Let $P$ be a special prime such that $uT$ is not nilpotent on $\EE{P}$. The set of special primes minimally containing
$P$ is finite.
\end{enumerate}
\end{cor}
\begin{proof}

The first statement follows from Lemma \ref{Lemma: Zero action submodule} with $e=1$, $\Theta=uT$: if $P\supseteq I_1(uR)$, the restriction of $uT$ to $\EE{P}$ is zero.

For $P$ as in (b), any special prime ideal $Q$ minimally containing $P$ is either among the finitely many
special primes not containing $I_1(uR)$ or a special prime which contains $I_1(uR)+P\supsetneq P$,
and in the latter case it is among the minimal primes of $I_1(uR)+P$.

For (c) note that if $uT$ is not nilpotent on $\EE{P}$, the restriction of $uT$ to $\EE{P}$ is not zero,
and Lemma \ref{Lemma: Zero action submodule}
shows that $P$ does not contain $I_1(uR)$
\end{proof}

A by-product of Theorem \ref{Theorem: the Utah Proof} and Corollary \ref{Corollary: finitely many minimally minimally containing primes}
is the algorithm described in
\cite[section 3]{Katzman-Schwede} which produces in the $F$-finite case
all $\Theta$-special primes $P$ for which the restriction of $\Theta$ to $\EE{P}$ is not the zero map.
As stated in the introduction, the aim of this paper is to extend this algorithm
and produce
the prime annihilators of submodules of $E^\alpha$ preserved by a given Frobenius map which restricts to a non-zero map,
and we shall do so in the subsequent sections.
It might be instructive at this point to see why Theorem \ref{Theorem: the Utah Proof} is not useful when $\alpha>1$: while parts (a) and (b) of the Theorem
hold in this extended generality, part (c) of the Theorem fails. The problem with (c) is that the module $\Hom_{S} (F_* S^\alpha, S^\alpha)$ is usually not
cyclic when $\alpha>1$.

\section{The case $\alpha>1$}
\label{Section: the case alpha>1}

The main aim of this section is to extend Corollary \ref{Corollary: finitely many minimally minimally containing primes}
to the case $\alpha>1$ and to obtain as a byproduct an algorithm for finding all special primes $P$ with the property that
for some $R[\Theta; f]$-submodule $M\subseteq E^\alpha$ with $(0:_R M)=P$,
the restriction of $\Theta$ to $M$ is not nilpotent.

\begin{thm}
\label{Theorem: the main theorem}
The set of all special primes $P$ with the property that
for some $R[\Theta; f]$-submodule $M\subseteq E^\alpha$ with $(0:_R M)=P$,
the restriction of $\Theta$ to $M$ is not zero, is finite.
\end{thm}

We will prove this theorem by induction on $\alpha$; the case $\alpha=1$ being the content of Corollary \ref{Corollary: finitely many minimally minimally containing primes}.
We shall assume henceforth in this section that $\alpha>1$ and that the theorem holds for $\alpha-1$ and that, additionally, as in the case $\alpha=1$,
there is an effective way of finding the finitely many special primes in question.

We should note that, given a non-zero Frobenius action
${U^\prime}^t T : \EE{ {W^\prime}} \rightarrow \EE{ {W^\prime}}$
with special prime $Q=(0 : R^{\alpha-1}/W^\prime)$,
the induction hypothesis gives us an effective method for finding this $Q$:
for any other ${U^\prime}^t T$-special prime $P\subset Q$,
$\left( P R^{\alpha-1} \right)^{\star {U^\prime}^t T} \subset \left( Q R^{\alpha-1} \right)^{\star {U^\prime}^t T} \subseteq W^\prime$
and hence the restriction of
${U^\prime}^t T$ to
$$\EE{ {\left( P R^{\alpha-1} \right)^{\star {U^\prime}}} }\supset
\EE{ {\left( Q R^{\alpha-1} \right)^{\star {U^\prime}}} } $$
is not nilpotent. Now we can enumerate all these special primes $P$
starting with $P=0$ and ascending recursively to bigger special primes until all such special primes are listed.

For the rest of this section, we will fix a Frobenius map $\Theta = U^t T : E^\alpha \rightarrow E^\alpha$ (where $U$ is a $\alpha \times \alpha$ matrix with entries in $R$) and we wish to find all the special primes with respect to $\Theta$.

The following lemma is our starting point of finding special primes $Q\supseteq P$ when a special prime $P$ is given.
\begin{lem}\label{Lemma: bootstrap with one element}
Let $Q$ be a special prime minimally containing the special prime $P$.
Let $a\in Q\setminus P$ and write $V={\left( (P+aR) R^{\alpha} \right)^{\star {U}}}$
then $Q$ is among the minimal primes of
$R^\alpha / V$.
\end{lem}
\begin{proof}
We have
$${\left( P R^{\alpha} \right)^{\star {U}}} \subseteq V \subseteq {\left( Q R^{\alpha} \right)^{\star {U}}}$$
and so
$$\EE{ {\left( Q R^{\alpha} \right)^{\star {U}}} } \subseteq
\EE{ V}  \subseteq \EE{ {\left( P R^{\alpha} \right)^{\star {U}}} }$$
and looking at the annihilators of these we get
$P \subseteq ( 0 :_R R^\alpha / V) \subseteq Q$
and we deduce that $Q$ contains a minimal prime of $R^\alpha / V$. This minimal prime is also special by Lemma \ref{Lemma: associated of special is special}
and since $Q$ minimally contains $P$, this minimal prime must equal $Q$.
\end{proof}

Next, we want to treat a (crucial) special case: the $\alpha$-th column of $U$ is entirely zero. To this end, we need the following lemma, which will enable us to reduce the rank of $U$ by one when we handle the aforementioned special case.

\begin{lem}\label{Lemma: reduce alpha}
Assume $\alpha>1$.
Let $Q$ be a special prime, and let $W\subseteq R^\alpha$ be such that $U W \subseteq W^{[p]}$ and $(0 :_R R^\alpha/ W)=Q$.
Let $a\notin Q$ and let $X$ be an invertible $\alpha \times \alpha$  matrix with entries in the localization $R_a$.
Let $\nu \gg 0$ be such that $U^\prime= a^\nu X^{[p]} U X^{-1}$ has entries in $R$ and let $W^\prime=X W_a \cap R^\alpha$.
Write $\Theta^\prime={U^\prime}^t T$.
Then
\begin{enumerate}
  \item [(a)]
  $Q$ is a minimal prime of $(0 :_R R^\alpha /W^\prime)$,
  \item [(b)]
  $U^\prime W^\prime\subseteq {W^\prime}^{[p]}$ and hence $Q$ is ${U^\prime}^t T$-special, and
  \item [(c)]
  if the restriction of $\Theta^e$ to $\EE{ W}$ is not zero,
  nor is the restriction ${\Theta^\prime}^e$ to $\EE{ {W^\prime}}$,
\end{enumerate}
\end{lem}
\begin{proof}

We have
$$(0 :_R R^\alpha /W^\prime)_a = (0 :_{R_a} R_a^\alpha /X W_a)= (0 :_{R_a} R_a^\alpha /W_a) = (0 :_R R^\alpha /W)_a = Q R_a$$
and (a) follows.


For (b) consider the commutative diagram
\begin{equation*}\label{CD2}
\xymatrix{
R_a^\alpha/W_a \ar@{>}[r]^{U} \ar@{>}[d]^{X}              & R_a^\alpha/ W_a^{[p]} \ar@{>}[d]^{X^{[p]}}   \\
R_a^\alpha/ XW_a \ar@{>}[r]^{X^{[p]} U X^{-1}}            & R_a^\alpha/ X^{[p]} W_a^{[p]} \\
}
\end{equation*}
and compute
$$U^\prime W^\prime = a^\nu X^{[p]} U X^{-1} (X W_a \cap R^\alpha) \subseteq ( a^\nu X^{[p]} U X^{-1} X W_a) \cap R^\alpha
\subseteq $$
$$(X^{[p]} W_a^{[p]}) \cap R^\alpha = (X W_a)^{[p]} \cap R^\alpha= (XW_a \cap R^\alpha)^{[p]}=  {W^\prime}^{[p]}. $$
The second statement in (b) now follows from (a) and Lemma \ref{Lemma: associated of special is special}.

Lemma \ref{Lemma: Zero action submodule} shows that the restriction of $\Theta^e$ to $\EE{ W}$ is not
zero if and only if
$$I_e ( U^{[p^{e-1}]} U^{[p^{e-2}]} \cdots U  R^\alpha) \not\subseteq W, $$
i.e., if and only if
$$U^{[p^{e-1}]} U^{[p^{e-2}]} \cdots U  R^\alpha \not\subseteq W^{[p^e]} .$$
The same Lemma shows that to prove (c) we need to verify that
this implies that
$$I_e( {U^\prime}^{[p^{e-1}]} {U^\prime}^{[p^{e-2}]} \cdots {U^\prime} R^\alpha) \not\subseteq W^\prime, $$
 i.e., that
$${U^\prime}^{[p^{e-1}]} {U^\prime}^{[p^{e-2}]} \cdots {U^\prime} R^\alpha \not\subseteq {W^\prime}^{[p^e]} .$$

Write $b=a^\nu a^{p \nu} \cdots a^{p^{e-1} \nu}$.
We calculate
$${U^\prime}^{[p^{e-1}]} {U^\prime}^{[p^{e-2}]} \cdots {U^\prime}=
b X^{[p^e]} U^{[p^{e-1}]} U^{[p^{e-2}]} \cdots U  X^{-1}$$
and if
$b X^{[p^e]} U^{[p^{e-1}]} U^{[p^{e-2}]} \cdots U  X^{-1} R^\alpha \subseteq {W^\prime}^{[p^e]}$
we may localize at $a$ to obtain
\begin{eqnarray*}
X^{[p^e]} U^{[p^{e-1}]} U^{[p^{e-2}]} \cdots U  R_a^\alpha & = & b X^{[p^e]} U^{[p^{e-1}]} U^{[p^{e-2}]} \cdots U  X^{-1}  R_a^\alpha \\
& \subseteq & {W_a^\prime}^{[p^e]} \\
& = & {(X W_a)}^{[p^e]} \\
& = & X^{[p^e]} W_a^{[p^e]}
\end{eqnarray*}
hence $U^{[p^{e-1}]} U^{[p^{e-2}]} \cdots U R_a^\alpha  \subseteq {W_a}^{[p^e]}$, and since $a$ is not a zero-divisor on $R^\alpha/W^{[p^e]}$ we deduce
$U^{[p^{e-1}]} U^{[p^{e-2}]} \cdots U R^\alpha  \subseteq {W}^{[p^e]}$, contradicting our assumption.

\end{proof}

We are now in position to deal with the following crucial special case.

\begin{prop}\label{Proposition: Zero column}

Assume that the $\alpha$th column of $U$ is zero. Then Theorem \ref{Theorem: the main theorem} holds and
there exists an effective method for finding the special primes $P$ with the property that
for some $R[\Theta; f]$-submodule $M\subseteq E^\alpha$ with $(0:_R M)=P$,
the restriction of $\Theta$ to $M$ is not zero.
\end{prop}

\begin{proof}
It suffices to show that, when given a special prime $P$, we can always find all special primes $Q$ that minimally contains $P$ (since we can always start with the special prime $(0)$). To this end, assume that $P$ is a special prime.

Let $\pi : R^\alpha \rightarrow  R^{\alpha-1}$ be the projection onto the first $\alpha-1$ coordinates,
let $U_0$ be the submatrix of $U$ consisting of its first $\alpha-1$ rows and columns. Define the Frobenius map $\Theta_0 : E^{\alpha-1} \rightarrow E^{\alpha-1}$ given by $\Theta_0=U_0^t T$.
Let $Q$ be a special prime minimally containing $P$, and  let $W= {Q R^\alpha}^{\star U}$ so that
$U W \subseteq W^{[p]}$ and $(0:_R R^\alpha/W)=Q$.

Our proof consists of a number of steps.

\noindent {\bf (a)}
$\EE{P R^\alpha}$ being an $R[\Theta; f]$-module is equivalent to
$P \Image U \subseteq P^{[p]} R^\alpha$, and this implies
that all entries in $U$ are in $(P^{[p]} : P)$. This shows that $\EE{ P R^{\alpha-1} }$ is an $R[\Theta_0; f]$-module and that $P$ is $\Theta_0$-special.

\noindent {\bf (b)}
Consider the case when the action of $\Theta_0$ on $\EE{P R^{\alpha-1}}$ is nilpotent. Pick $e\geq 1$ so that the restriction of $\Theta_0^e$ to $\EE{P R^{\alpha-1}}$ is zero.
Consider the matrix
$U^{[p^{e-1}]} U^{[p^{e-2}]} \cdots U$: denote its last row  $(g_1, \dots, g_{\alpha-1}, 0)$ and note that
its top left $(\alpha-1) \times (\alpha-1)$ submatrix is
$U_0^{[p^{e-1}]} U_0^{[p^{e-2}]} \cdots U_0$
and our assumption implies that the entries of this matrix are in  $P^{[p^e]} \subset Q^{[p^e]}$
so the action of $\Theta^e={U^{[p^{e-1}]}}^t {U^{[p^{e-2}]}}^t \cdots U^t$ on $\EE{W}$ is the same
as the action of a matrix $U_e$ whose first $\alpha-1$ rows are zero and its last row is
$(g_1, \dots, g_{\alpha-1}, 0)$.

Define $L$ as the union of
\begin{eqnarray*}
L_0 & = & Q R^\alpha\\
L_1 & = & I_e (U_e Q R^\alpha)  + Q R^\alpha \\
L_2 & = & I_e \left( U_e I_e (U_e Q R^\alpha)  + U_e Q R^\alpha \right) + I_e (U_e Q R^\alpha)  + Q R^\alpha= I_e (U_e Q R^\alpha)  + Q R^\alpha\\
\end{eqnarray*}
and the stable value at $L_1$
defines an $R[U_e T^e; f^e]$-module $\EE{L_1}$ whose annihilator is $Q$.
Now $U_e Q \subseteq Q^{[p]}$ so
$g_i Q \subseteq Q^{[p^e]}$ for all $1\leq i\leq \alpha-1$, hence $Q$ is $g_i T^e$-special for all $1\leq i\leq \alpha-1$.
One of these $g_i T^e$ must restrict to a non-zero map on $\EE{P}$ otherwise $g_i\in P^{[p^e]}$ for all $1\leq i\leq \alpha$ and
then the restriction of $\Theta^e$ to $\EE{P R^\alpha}$ would be zero.
We can now find all such $Q$ using the algorithm in section 5 of \cite{Katzman-Schwede}. This finishes our step {\bf (b)}.

Let $\tau\subset R$ be the intersection of the finite set of $\Theta_0$-special prime ideals minimally containing $P$ and
write the submodule
$\Nil(E^{\alpha-1}):=\{ z\in E^{\alpha-1} \,|\, \Theta_0^e z=0 \text{ for some } e\geq 0\}$
as $\EE{K}$ where $K\subseteq R^{\alpha-1}$.
Write $J=(0 :_R R^{\alpha-1}/\pi(W))$. Note that $J\supseteq Q$ and hence $J\supsetneq P$.

\noindent {\bf (c)}
Let $K_0=R^{\alpha-1}$ and define recursively $K_{j+1}=I_1( U_0 K_j)$ for all $j\geq 0$. Then clearly $K\subseteq K_0$.
Lemma \ref{Lemma: I1(uK)} implies that $I_1(UK)=K$, so if we assume inductively that $K\subseteq K_j$, then
$K=I_1(UK) \subseteq I_1(U K_j)=K_{j+1}$.

\noindent {\bf (d)}
We claim that $\tau K \subseteq (J R^{\alpha-1})^{\star U_0}$ and hence that $\tau K \subseteq \pi(W)$; and we reason as follows. Lemma \ref{Lemma: associated of special is special} shows that $\tau\subseteq \sqrt{J}$, and so for all large  $e\geq 0$  we have
$\tau^{[p^e]} \subseteq J$ and hence also
$$ \left( \tau^{[p^e]} K \right)^{\star U_0} \subseteq \left( J R^{\alpha-1} \right)^{\star U_0}  \subseteq \pi(W)^{\star U_0} =\pi(W) .$$

We compute $\left( \tau^{[p^e]} K \right)^{\star U_0}$ as the union of
\begin{eqnarray*}
L_0         & = & \tau^{[p^e]} K \\
L_1         & = & I_1\left( U_0 \tau^{[p^e]} K \right) + L_0 =   \tau^{[p^{e-1}]} I_1\left( U_0 K \right) + L_0  = \tau^{[p^{e-1}]} K_1 + L_0 \\
L_2         & = & I_1\left( U_0 \tau^{[p^{e-1}]} K_1 \right) + L_1 =   \tau^{[p^{e-2}]} I_1\left( U_0 K_1 \right) + L_1 = \tau^{[p^{e-2}]} K_2 + L_1\\
            & \vdots & \\
L_e         & = & \tau K_e + L_{e-1} \\
            & \vdots & \\
\end{eqnarray*}
and from (c) we deduce that
$\tau K \subseteq \tau K_e \subseteq L_e  \subseteq  \left( \tau^{[p^e]} K \right)^{\star U_0}$.

\noindent {\bf (e)}
If the action of $\Theta_0$ on $\EE{P R^{\alpha-1}}$ is not nilpotent, then we claim that $\tau K \not\subseteq P R^{\alpha-1}$ and $U_0 \tau K \not\subseteq P^{[p]} R^{\alpha-1}$. The action of $\Theta_0$ on $\EE{P R^{\alpha-1}}$ being not nilpotent
is equivalent to $K \not\subseteq P R^{\alpha-1}$. Since $\tau\not\subseteq P$ we obtain
$\tau K \not\subseteq P R^{\alpha-1}$.
If $U_0 \tau K \subseteq P^{[p]} R^{\alpha-1}$, then  $U_0 K \subseteq P^{[p]} R^{\alpha-1}$,
$K=I_1(U_0 K)\subseteq P R^{\alpha-1}$, and the action of $\Theta_0$ on $\EE{P R^{\alpha-1}}$ is nilpotent. This completes our step {\bf (e)}.

For any $v=(w_1, \dots, w_{\alpha-1}, w_\alpha)^t \in R^\alpha$, we define $\overline{w}=(w_1, \dots, w_{\alpha-1}, 0)^t$
and for any $V\subseteq R^\alpha$ let $\overline{V}$ denote $\{ \overline{v} \,|\, v\in V\}$.
Let $\iota : R^{\alpha-1} \rightarrow R^{\alpha-1} \oplus R$ be the natural inclusion $\iota(v)=v\oplus 0$.
Note that $\overline{V}=\iota (\pi(V))$.

\noindent {\bf (f)}
We claim $I_1 \left(U \iota (\tau K) \right)^{\star U}\subseteq W$ and $I_1 \left(U \iota (\tau K) \right)^{\star U}\not\subseteq P R^\alpha$.
  Define $W_1=\{ w\in W \,|\, \pi(w)\in \tau K \}$ and note that (d) implies that $\pi(W_1)=\tau K$.
  We have $W_1 \subseteq W$ hence $W_1^{\star U} \subseteq W^{\star U}=W$; also
  $W_1^{\star U} = I_1( U W_1)^{\star U} + W_1$
  and $U W_1=U \overline{W_1}=  U \iota (\tau K)$ hence
  $I_1 \left((U \iota (\tau K) \right)^{\star U}\subseteq W_1^{\star U}  \subseteq W$.

If $I_1 \left(U \iota (\tau K) \right)^{\star U} \subseteq P R^\alpha$
then $U_0\tau K=\pi\left(U \iota (\tau K)\right)  \subseteq P^{[p]} R^{\alpha-1}$, in contradiction to {\bf (e)}.

\noindent {\bf (g)}
Let $M^\prime$ be a matrix whose columns generate
$I_1 \left((U \iota (\tau K) \right)^{\star U}\subseteq   W$
and choose an entry $a$ in it which is not in $P$.

If  $a\in Q$,  Lemma \ref{Lemma: bootstrap with one element} shows that
$Q$ is among the minimal primes of  $\left( (P+R a) R^\alpha \right)^{\star U}$, and we are done (with finding such $Q$).

If $a\notin Q$, we can apply Lemma \ref{Lemma: reduce alpha} with the matrix $X$ with entries in $R_a$ such that
$\mathbf{e}_\alpha \in W^\prime=X W_a \cap R$
where $\mathbf{e}_\alpha$ is the vector $(0,0, \dots, 0, 1)^t\in R^\alpha$.
Now $R^\alpha/W^{\prime} \cong R^{\alpha-1}/\pi(W^{\prime})$,
$Q$ is an associated prime of $\left(0 :_R R^{\alpha-1}/\pi(W^{\prime}) \right)$,
and $Q$ is ${U^\prime}^t T$-special, with $U^\prime$ as defined in Lemma \ref{Lemma: reduce alpha}.

Now $Q$ is special for the Frobenius map obtained by the restriction of ${U^\prime}^t T$
to $\EE{W^{\prime}}$; part (b) and Lemma \ref{Lemma: reduce alpha}(c) shows that this map is not nilpotent
and we can apply the induction hypothesis to find $Q$. This finishes our last step {\bf (g)} and hence the proof of our proposition.
\end{proof}

Our next theorem provides an effective algorithm to find all special primes with the property that for some $R[\Theta; f]$-submodule $M\subseteq E^\alpha$ with $(0:_R M)=P$, the restriction of $\Theta$ to $M$ is not zero.

\begin{thm}\label{Theorem: Main Theorem}
Let $P$ be a special prime, let $Q$ be a special prime minimally containing $P$.
Let $M$ be a matrix whose columns generate $\left( P R^\alpha \right)^\star$.

\begin{enumerate}

  \item[(I)] Assume that $\Image M \supsetneq P R^\alpha$. Then either
  \begin{enumerate}
    \item[(a)] All entries of $M$ are in $Q$, hence there is such an entry $q\in Q\setminus P$, and $Q$ is among the minimal primes of
    $(0 :_R  R^\alpha/ ((P+qR) R^\alpha)^{\star})$, or
    \item[(b)] There exists an entry of $M$ which is not in $Q$, and $Q$ is a special prime of a Frobenius action on $E^{\alpha-1}$.
  \end{enumerate}

  \item[(II)] Assume that $\Image M = P R^\alpha$. There exist an $a_1\in R\setminus P$, a $g\in (P^{[p]}:P)$ and an $\alpha\times\alpha$ matrix $V$ such that
  for some $\mu>0$, $a_1^\mu U\equiv g V$ modulo $P^{[p]}$. Write $d=\det V$.
  Either
  \begin{enumerate}
    \item[(a)] $d\in P$, and $Q$ can be obtained as a special prime of a Frobenius action on $E^{\alpha-1}$,

    \item[(b)] $d\in Q\setminus P$, and $Q$ is among the minimal primes of
    $(0 :_R  R^\alpha/ ((P+dR) R^\alpha)^{\star})$, or

    \item[(c)] $d\notin Q$ and $Q$ can be obtained as a special prime of the Frobenius action on $gT$ on $E$.
  \end{enumerate}

\end{enumerate}
\end{thm}

\begin{proof}
Choose $W_Q\subseteq R^\alpha$ such that $U W_Q \subseteq W_Q^{[p]}$ and such that
$(0 :_R R^\alpha/ W_Q)=Q$.

Assume first that we are in case (I).
If (a) we can choose an entry $q$ of $M$ such that $q\in Q\setminus P$.
Now  Lemma \ref{Lemma: bootstrap with one element} shows that $Q$ is among the minimal primes of
$\left(0:_R R^\alpha / \left((P+qR) R^\alpha\right)^\star\right)$.

Assume now that we are in case (I)(b), i.e., assume the existence of an entry of $M$ not in $Q$.
Note that with $W_Q$ as above we must have $W_Q \supseteq \left( Q R^\alpha \right)^\star\supseteq \left( P R^\alpha \right)^\star=\Image M$
and if we choose a matrix $M_Q$ whose columns generate $W_Q$, we see that $M_Q$ contains an entry not in $Q$.
We now apply Lemma \ref{Lemma: reduce alpha} with $W$ replaced by $W_Q$: there exists an invertible $\alpha\times \alpha$ matrix $X$ with entries in $R_a$
such that $X W_Q$ contains the elementary vector $e_\alpha:=(0, \dots, 0, 1)^t$, and with $U^\prime$ and $W^\prime$ as in the lemma,
we obtain $Q$ as a special prime of the Frobenius action ${(U^\prime)}^{t}$ on $\EE{W^\prime}$.
We note that $R^\alpha/W^\prime\cong R^{\alpha-1}/W^{\prime\prime}$ where $W^{\prime\prime}$ is the projection of $W^{\prime}$ onto its first
$\alpha-1$ coordinates, hence $Q$ is a special prime of the Frobenius action ${(U^\prime)}^{t}$ on $\EE{W^{\prime\prime}}$.
We may now apply the induction hypothesis and deduce that we have an effective method of finding $Q$.

Assume henceforth case (II) and note that $U P R^\alpha \subseteq P^{[p]} R^\alpha$ implies that the entries of $U$ are in $(P^{[p]} : P)$.
Write $S=R/P$; recall that $(P^{[p]} : P)/P^{[p]} \cong \Hom_S( F_* S, S)$ is an $S$-module of rank one, so we can find an element $a_1\in R\setminus P$
such that the localization of $(P^{[p]} : P)/P^{[p]}$ at $a_1$ is generated by one element $g/1+P_{a_1}^{[p]}$ as an $S_{a_1}$-module (and hence also as an $R_{a_1}$-module).
If $a_1\in Q$, we may construct $Q$ as in case (I)(a), so assume $a_1\notin Q$.

We can now write $a_1^\mu U=g V + V^\prime$ for some $\mu\geq 0$ and $\alpha\times\alpha$ matrices $V$ and $V^\prime$ with entries in $R$ and
$P^{[p]}$, respectively. Now the restriction of the Frobenius map  ${V^\prime}^t T$ to $\EE{P R^\alpha}$ is zero,
and hence so is its restriction to  $\EE{ W_Q}$, and we may, and do, replace $V^\prime$ with the zero matrix without affecting any issues.

Write $d=\det V$  and distinguish between three cases:
\begin{enumerate}
  \item[(1)] Assume $d\in P$. Working in the fraction field $\mathbb{F}$ of $S$ we can find an invertible matrix $X$ with entries in $\mathbb{F}$
  such that the last column of $V X^{-1}$ is zero. We can now find an element $a_2\in R\setminus P$ such that the entries of $X$ and $X^{-1}$ are in
  the localization $R_{a_2}$; write $a=a_1 a_2$.

  We now apply Lemma \ref{Lemma: reduce alpha} to deduce, using the Lemma's notation, that $Q$ is a special prime of the Frobenius map ${U^\prime}^t T$ where
  $U^\prime=a^\nu X^{[p]} U X^{-1}$, $Q=(0 :_R R^\alpha / W^\prime)$, and
  $(U^\prime)^t T : \EE{W^\prime} \rightarrow  \EE{W^\prime}$ is not nilpotent.
  We note that the last column of $U^\prime$ is zero and hence we can produce $Q$ using Proposition \ref{Proposition: Zero column}.

  \item[(2)]
  If $d\in Q\setminus P$, we may construct $Q$ as in case (I)(a).

  \item[(3)]
  Finally we may assume that $d\notin Q$.
  We now apply Lemma \ref{Lemma: reduce alpha} with $a=a_1 d$, $W=\left( Q R^\alpha \right)^{\star U}$ and $X=I_\alpha$, the
  the $\alpha \times \alpha$ identity matrix. With the notation of that Lemma, we have $W^\prime = \left( Q R^\alpha \right)^{\star U}_a \cap R^\alpha$.
  We now explicitly compute $\left( Q R^\alpha \right)^{\star U}_a$ as the union of the sequence
  \begin{eqnarray*}
  L_0 & = & Q R^\alpha_a \\
  L_1 & = & I_1( U Q R^\alpha_a ) + Q R^\alpha_a = I_1( g V Q R^\alpha_a ) + Q R^\alpha_a = I_1( g Q R^\alpha)_a + Q R^\alpha_a\\
  L_2 & = & I_1( g L_1) + L_1\\
  &\vdots&
  \end{eqnarray*}
  and we compare this to $\left( Q R^\alpha \right)^{\star g I_\alpha}_a$ explicitly computed  as the union of the sequence
  \begin{eqnarray*}
  L_0^\prime & = & Q R^\alpha_a \\
  L_1^\prime & = & I_1( g Q R^\alpha_a ) + Q R^\alpha_a = I_1( g Q R^\alpha )_a + Q R^\alpha_a\\
  L_2^\prime & = & I_1( g L_1^\prime) + L_1^\prime\\
  &\vdots&
  \end{eqnarray*}
  where we used Lemma \ref{Lemma: Ie and localization} in the third equalities for $L_1$ and $L_1^\prime$.
  The fact that $L_1=L_1^\prime$ implies that $L_i=L_i^\prime$ for all $i\geq 1$ and hence
  $\left( Q R^\alpha \right)^{\star U}_a=\left( Q R^\alpha \right)^{\star g I_\alpha}_a$.
  Now $W^\prime = \left( Q R^\alpha \right)^{\star U}_a \cap R^\alpha=\left( Q R^\alpha \right)^{\star g I_\alpha}_a \cap R^\alpha$,
  Lemma \ref{Lemma: reduce alpha} implies that $Q$ is a minimal prime of $R^\alpha/W^\prime$ and hence
  $Q$ is $(g I_\alpha)T$-special.

  We can now deduce that $Q$ is $gT$-special, and we find $Q$ using the case $\alpha=1$, provided that
  $P$ is also $gT$-special and $g T: \EE{ P} \rightarrow \EE{P}$ is non zero. The former follows from the fact that $g\in (P^{[p]}:P)$
  and the latter from the fact that the image of $g+P^{[p]}$ after localization at the fraction field of $S$ generates a non-zero module,
  hence $g\notin P^{[p]}$.

\end{enumerate}
\end{proof}

Theorem \ref{Theorem: Main Theorem} also finishes the induction step of the proof of Theorem \ref{Theorem: the main theorem}. For the sake of completeness, we end with a proof of Theorem \ref{Theorem: the main theorem}.

\begin{proof}[Proof of Theorem \ref{Theorem: the main theorem}]
We will use induction on $\alpha$. When $\alpha=1$, our theorem has been proved in Section 4. Assume that $\alpha>1$ and our theorem has been established for $\alpha-1$. Let $P$ be a special prime (we can always start with $P=(0)$) and let $M$ be a matrix whose columns generate $\left( P R^\alpha \right)^\star$, then there will be two cases: (I). $\Image M \supsetneq P R^\alpha$; (II). $\Image M = P R^\alpha$. As proved in Theorem \ref{Theorem: Main Theorem}, in either cases there are only finitely many special primes $Q$ minimally containing $P$ (by our induction hypothesis). Since only finitely many special primes are produced at each step and the number of steps is bounded by the dimension of $R$, there are only finitely special primes with the desired property.
\end{proof}

\section{The algorithm in action and two calculations}
\label{Section: The algorithm in action and two calculations}

We first piece together all the results of the previous sections into an explicit algorithm.

\begin{description}
\item{\bf Input}
\begin{itemize}
\item A ring $R=\mathbb{K}[x_1, \dots, x_n]$ where $\mathbb{K}$ is a field of prime characteristic $p$, and
\item A $\alpha\times \alpha$ matrix $U$ with entries in $R$ such that $U^t T$ is not a nilpotent Frobenius map on $E^\alpha$.
\end{itemize}

\item{\bf Output}
\begin{itemize}
The list $\mathcal{A}$ of all $U^t T$-special primes $Q$ with the property that the
Frobenius map on $\EE{Q R^\alpha}$ is not nilpotent.
\end{itemize}

\item{\bf Initialize}\\
$\mathcal{A}=\{ 0 \}$ , $\mathcal{B}=\emptyset$.

\item{\bf Execute the following}\\
If $\alpha=1$ use the algorithm described in \cite{Katzman-Schwede} to find the desired
special primes, put these in $\mathcal{A}$, output it, and stop.

While $\mathcal{A}\neq \mathcal{B}$, pick any $P\in \mathcal{A}\setminus \mathcal{B}$,
write $(P R^\alpha)^{\star U}$ as the image of a matrix $M$ and do the following:
\begin{enumerate}
  \item[(1)] If there is an entry $a$ of $M$ which is not in $P$ then
\begin{enumerate}
\item[(1a)] add to  $\mathcal{A}$ the minimal primes of the annihilator of $R^\alpha /((P+aR) R^\alpha)^{\star U}$, and
\item[(1b)]
find an $\alpha \times \alpha$ invertible matrix $X$ with entries in $R_a$ such that $\Image M R_a$
contains the $\alpha$th elementary vector,
choose $\nu\gg 0$ such that
$U^\prime = X^{[p]} U X^{-1}$ has entries in $R$,
let $U_0$ be the submatrix of $U$ consisting of its first $\alpha-1$ rows and columns,
find recursively the special primes of $U_0^T T$,
add to $\mathcal{A}$ those which are also $U^T T$-special.
\end{enumerate}

  \item[(2)] If  $(P R^\alpha)^{\star U}=(P R^\alpha)$ find $a_1\in R\setminus P$, $g\in (P^{[p]}:P)$,
  $\alpha\times \alpha$ matrix $V$ and $\mu>0$ such that $a_1^\mu \equiv g V$ modulo $P^{[p]}$.
  Compute $d=\det V$.
\begin{enumerate}
\item[(2a)] If $d\in P$, find an element $a_2\in R\setminus P$ and an invertible matrix with entries in $R_{a_2}$
such that the last column of $UX^{-1}$ is zero.
Find $\nu\gg 0$ such that the entries of $U_1(a_1 a_2)^\nu X^{[p]} U X^{-1}$ are in $R$.
Let $U_0$ be the submatrix of $U_1$ consisting of its first $\alpha-1$ rows and columns.
\begin{enumerate}
  \item[(2a)(i)] If the restriction of $(U_0^t T)^e$ to $\EE{P R^{\alpha-1}}$ is zero for some $e\geq 0$,
  write the last row of $U_1^{[p^e]} U_1^{[p^{e-1}]} \dots U_1$ as $(g_1, g_2, \dots, g_{\alpha-1}, 0)$, find all
  $g_i T$-special primes as in \cite{Katzman-Schwede} and add to $\mathcal{A}$ those which are also
  $U^t T$-special.
  \item[(2a)(ii)] If the restriction of $U_0^t T$ to $\EE{P R^{\alpha-1}}$ is not nilpotent,
compute recursively all $U_0^t T$-special primes minimally containing $P$ and their intersection $\tau$,
find $K\subseteq R^{\alpha-1}$ such that $\EE{K}$ is the module of $U_0^t T$-nilpotent elements,
compute $I_1\left( U_1 \iota (\tau K)\right)^{\star U_1}$ and write this as the image of a matrix $M^\prime$.
Find a entry $a$ in $M^\prime$ not in $P$. Now
\begin{itemize}
\item Add to $\mathcal{A}$ all the minimal primes of $R^\alpha/(P+aR)^{\star U_1}$ which are also
$U^t T$-special.
\item Find an invertible matrix $X$ with entries in $R_a$ such that $\Image M^\prime R_a$ contains
the $\alpha$th elementary vector. Find $\nu\gg  0$ such that $U_2=a^\nu X^{[p]} U_1 X^{-1}$ has entries in $R$.
Let $U_3$ be the submatrix of $U_2$ consisting of its first $\alpha-1$ rows and columns.
Compute recursively all $U_3^t T$-special primes and add those which are also $U^t T$-special to $\mathcal{A}$.
\end{itemize}
\end{enumerate}

\item[(2b)] If $d\notin P$, add to  $\mathcal{A}$ the minimal primes of the annihilator of $R^\alpha ((P+dR) R^\alpha)^{\star U}$.
\item[(2c)] If $d\notin P$, use the algorithm described in \cite{Katzman-Schwede} to find the $gT$-special primes and add to $\mathcal{A}$
those which are also $U^t T$-special.
\end{enumerate}

\item[(3)] Add $P$ to  $\mathcal{B}$.
\end{enumerate}

Output $\mathcal{A}$ and stop.
\end{description}

We now apply the algorithm to the calculation of special primes in two examples.
The first, illustrates the trick of reducing $\alpha=2$ to a calculation with $\alpha=1$, and the second
illustrates a case where Theorem \ref{Theorem: Main Theorem}(II)(a), and hence, Proposition \ref{Proposition: Zero column}, needs to be applied.

\subsection{First example}
Let
$R=\mathbb{Z}/2\mathbb{Z}[x,y,z]$ and let
$$U=\left( \begin{array}{cc} x^3+y^3+z^3 & xy^2z^5 \\ x(y^2 + z^2) & x^3 \end{array} \right)$$

We start with the special prime $P_0=0$:
to find the special primes minimally containing $P_0$, we apply Theorem \ref{Theorem: Main Theorem} and find ourselves in
case (II) with $g=1$, $U=V$, and $d=\det U=x^2(x(y^3+z^2)+x^4+y^2z^5(y^2+z^2))$.
We look for special primes containing $d$ as in Theorem
\ref{Theorem: Main Theorem}(II)(b):
$$ \left({dR^2}\right)^{\star U}=\Image
\left( \begin{array}{cccc}
y & z & 0 & x\\
0 & 0 & x & y+z\\
\end{array}
\right).
$$
The annihilator of  $R^2/\left({dR^2}\right)^{\star U}$
has a unique minimal prime
$P_1=(x, y+z) R$, hence $P_1$ is a special prime.
We look for special primes not containing $d$ as in 
Theorem \ref{Theorem: Main Theorem}(II)(c):
these would contain $g=1$, hence there aren't any.

Next we find special primes $Q$ containing $P_1$.

We compute
$$\left(P_1{R^2}\right)^{\star U}=
\Image
\left( \begin{array}{ccccc}
x & y & z & 0 & 0\\
0 & 0 & 0 & x & y+z\\
\end{array}
\right) \supsetneq P_1 R^2
$$
and we are in case (I) of Theorem  \ref{Theorem: Main Theorem}.
We first consider the cases $y\in Q$ and $z\in Q$ which give:
$$\left((x,y,z) R^2\right)^{\star U}=(x,y,z) {R^2}$$ and we obtain the special prime
$P_2=(x,y,z)R$.
Assume now that $y\notin Q$,  let
$$X=
\left( \begin{array}{cc}
1/y & 0 \\ 0 & 1\\
\end{array}
\right)
$$
and compute
$$ U^\prime = y X^{[2]} U X^{-1}=
\left( \begin{array}{cc} x^3+y^3+z^3 & xyz^5 \\ x y^2 (y^2 + z^2) & x^3 y \end{array} \right).$$
We now deduce that $Q$ must be a $x^3 y$-special prime; these are computed with the algorithm in \cite{Katzman-Schwede}
to be
$xR$, $yR$ and $(x,y)R$.
We compute
$\Ann_R R/ (x{R^2})^{\star U}=P_1$,
$\Ann_R R/ (y{R^2})^{\star U}=P_2$, and
$\Ann_R R/ ((x,y){R^2})^{\star U}=P_2$
so $xR$, $yR$ and $(x,y)R$ are not $U$-special.

We conclude that the set of special primes is $\{ P_0, P_1, P_2 \}$.

\subsection{Second example}

Let $R=\mathbb{Z}/2\mathbb{Z}[x,y,z]$, $f=x^3+y^3+z^3$, $g=x^2+z^4$ and define
$$U=\left( \begin{array}{cc} x f & y f \\ x g & y g \end{array} \right).$$

We start with the special prime $P_0=0$, and find the special primes $Q$ minimally containing $P_0$
by following
Theorem \ref{Theorem: Main Theorem}(II)(a) as follows:
with
$$X^{-1}=
\left( \begin{array}{cc} 1 & -y/x \\ 0 & 1 \end{array} \right)
$$
we compute
$$U^\prime=x X^{[2]} U X^{-1} =\left( \begin{array}{cc} x^2 f + y^2 g& 0 \\ x^2 g & 0 \end{array} \right) .$$
The $U^t$-special primes
either contain $x$ or are ${U^\prime}^t$-special primes.
To find the former we find that $P_1=(x,z)R$ is the only minimal prime of the annihilator of
$R^2/{(x R^2)}^{\star U}$ and we add it to the list of $U^t$-special primes.

We now find the ${U^\prime}^t$-special primes using Proposition \label{Proposition: Zero column}
as follows.
Write $U_0=(x^2 f + y^2 g)$ and compute $K={R}^{\star U_0}=(x,y)R \cap (x,z^2)R \cap (x^3,y,z)R$
and since this is not zero, $U_0 T$ is not nilpotent on $\EE{P_0}$ and we can proceed to find the
${U^\prime}^t$-special primes using Proposition \ref{Proposition: Zero column}(g).
We compute the set of $U_0 T$-special primes to be
$\{ \tau R, (x,y)R, (x,z)R \}$ where $\tau =(y^2 z^4+x^2(x^3 y^3+z^3+y^2))R$;
the intersection of these is
$\tau R$.
We now compute
$$I_1 \left( U \iota(\tau K)\right)^{\star U_0}=
\Image
\left( \begin{array}{ccccc}
 y & z & x & 0\\
 0 & 0 & y+z & x\\
\end{array}
\right);
$$
We now look for special primes which contain $x$, and were considered above, and those $Q$ which do not contain $x$.
For the latter we apply Lemma \ref{Lemma: reduce alpha} with $X$ being the identity, and
deduce that $Q$ is $U_0 T$-special which we found above. To test whether any of
$\tau R$,  $(x,y)R$, and  $(x,z)R$  is $U_0^t T$-special we compute
$(\tau R^2)^{\star U_0}=(x,y,z)R^2$,
$((x,y) R^2)^{\star U_0}=(x,y,z)R^2$, and
$((x,z) R^2)^{\star U_0}=(x,y,z)R^2$ so none of this is $U_0^t T$-special.
We conclude that the only $U_0^t T$-special is $0$.

We iterate the algorithm again, this time with the special $U^t T$-prime $P=P_1=(x,z)R$.
We compute
$$(P_1 R^2)^{\star U}=
\Image
\left( \begin{array}{ccccc}
x & y & z & 0 & 0\\
0 & 0 & 0 & x & z\\
\end{array}
\right) \supsetneq P_1 R^2
$$
and with Theorem \ref{Theorem: Main Theorem}(I) we look for special $U^t T$-primes minimally containing
$P_1$ among those containing $y$ and those not containing $y$.
The former must contain the minimal primes of the annihilator or
$$R^2/{(yR+P_1)}^{\star U}=R^2/(x,y,z)R^2$$
hence the only such special prime is  $P_2=(x,y,z)R.$
For those special primes which do not contain $y$, and application of Lemma \ref{Lemma: reduce alpha} (with $X$ as
the identity matrix) shows that that those special primes are $yg$ special.
These special primes can be computed to be
$y R$, $(y, x+z^2)R$ and $(x+z^2)$.
Only the last excludes $y$ and none contain $P_1$ so none yield new $U^tT$-special primes.

We conclude that the only $U^tT$-special primes are $P_0=0$, $P_1=(x,z)$ and $P_2=(x,y,z)$.

\section{Connections with Frobenius near splittings}
\label{Section: Connections with Frobenius near splittings}
Recall that a \emph{Frobenius near-splitting} of $R$ is an element $\phi$ of $\Hom_R (F_* R, R)$, and, if, additionally,
$\phi(1)=1$, we call $\phi$ a \emph{Frobenius splitting}. The results in \cite{Katzman-Schwede} used in this paper were mainly in terms of
ideals $I\subseteq R$ stable under a given Frobenius near-splitting, that is, ideals $I\subseteq R$ for which
$\phi(F_* I)\subseteq I$, and later in section 6.2 there a connection was established with submodules stable under a given Frobenius map on $E$.

Here we go the other way around and, having established a method for finding the radical annihilators of submodules of
$E^\alpha$, we show how this method finds the annihilators of certain modules stable under the following generalization of Frobenius near-splittings.

\begin{defi}
A \emph{Frobenius near-splitting} of $R^\alpha$ is an element $\phi$ of $\Hom_R (F_* R^\alpha, R^\alpha)$.
Given such a Frobenius near-splitting $\phi$, we call a submodule $V\subseteq R^\alpha$ $\phi$-compatible if
$\phi( F_* V) \subseteq V$.
\end{defi}

Thus if $V\subseteq R^\alpha$ is $\phi$-compatible we have a commutative diagram
\begin{equation*}
\xymatrix{
F_* R^\alpha  \ar@{>}[r]^{\phi} \ar@{>>}[d] & R^\alpha \ar@{>>}[d]\\
F_* R^\alpha / F_* V \ar@{>}[r]^{\phi}& R^\alpha/V\\
}
\end{equation*}
and if we take Matlis duals we obtain
\begin{equation}\label{eqn10}
\xymatrix{
\Hom(F_* R^\alpha, E)   & E^\alpha \ar@{>}[l]^{\phi^\vee}\\
\Hom(F_* R^\alpha / F_* V, E)   \ar@{^{(}->}[u] & \EE{V} \ar@{>}[l]^{\phi^\vee} \ar@{^{(}->}[u]\\
}
\end{equation}

The following establishes in the $F$-finite case
the connection between the annihilators of submodules of $E^\alpha$ fixed under a Frobenius map and the
annihilators of quotients of $R^\alpha$ by submodules compatible under a Frobenius near-splitting.

\begin{prop}
Assume that $R$ is $F$-finite. Let $\phi\in \Hom_R (F_* R^\alpha, R^\alpha)$ and let $V\subseteq R^\alpha$ be a $\phi$-compatible submodule.
Then $\phi^\vee$, the Matlis dual of $\phi$, is a Frobenius map on $E^\alpha$ with the property that
$\phi^\vee\left( \EE{V} \right) \subseteq  \EE{ V}$ and the annihilator of $R^\alpha/V$ coincides with that of
$\EE{V}$.

Hence the method of Theorem \ref{Theorem: the main theorem} finds
all radical annihilators of quotients $R^\alpha/V$ for $\phi$-compatible submodules $V$
for which $\phi(F_* R^\alpha)\not\subseteq V$.

\end{prop}
\begin{proof}

Since $R$ is $F$-finite,  $F_* R$ is a free $R$-module.
In this case one has a natural isomorphism
$\Hom_R(F_* R, E) \cong F_* E$ and diagram (\ref{eqn10}) is identified with
\begin{equation}
\xymatrix{
F_* E^\alpha   & E^\alpha \ar@{>}[l]^{\phi^\vee}\\
\Ann_{F_* E^\alpha}  F_* V^t   \ar@{^{(}->}[u] & \Ann_{E^\alpha} V^t \ar@{>}[l]^{\phi^\vee} \ar@{^{(}->}[u]\\
}
\end{equation}

We now recall that $R^\alpha/V$ and its Matlis dual have the same annihilator and to
establish the last claim we note that the restriction of $\phi^\vee$ to $\Ann_{E^\alpha} V^t$ is zero
precisely when the map  $F_* R^\alpha/ F_* V \rightarrow R^\alpha/ V$ vanishes, i.e.,
when $\phi(F_* R^\alpha)\subseteq V$.

\end{proof}

The correspondence between Frobenius near-splittings and Frobenius maps in the non-$F$-finite case may be far more complicated
(for example, cf.~section 4 in \cite{Katzman2}.)

We shall assume for the rest of this section that $R$ is $F$-finite and we will exhibit a more explicit connection between Frobenius map and near-splittings as follows.
We can choose a free basis for $F_* R=\mathbb{K}[\![x_1, \dots, x_d]\!]$ over $R$ which contains $F_* x_1^{p-1} \cdots x_d^{p-1}$ and
we let $\pi\in\Hom_R(F_* R, R)$ denote the projection onto the free summand $F_* x_1^{p-1} \cdots x_d^{p-1}$.
\begin{prop}\label{Proposition: explicit from of near-splittings}
Let $\alpha$ be a positive integer and write $\Phi\in \Hom_R(F_* R^\alpha, R^\alpha)$ for the direct sum of $\alpha$ copies of $\pi$.
\begin{enumerate}
\item[(a)] Any  $\phi\in \Hom_R(F_* R^\alpha, R^\alpha)$ has the form $\phi=\Phi \circ F_* U$ where $U$ is an $\alpha\times \alpha$ matrix
with entries if $R$.
\item[(b)] For any $\phi=\Phi \circ F_* U$ as in (a) and any $F_* R$-submodule $F_* V\subseteq F_* R^\alpha$,
$\phi(F_* V)=I_1 (U V)$.
\end{enumerate}
\end{prop}
\begin{proof}
Since $\Hom_R(F_* R^\alpha, R^\alpha)=\Hom_R(F_* R, R)^{\alpha \times \alpha}$ we need to show (a) holds for $\alpha=1$.
Let $\mathcal{B}$ a $\mathbb{K}$-basis for $F_* \mathbb{K}$ which contains $1\in \mathbb{K}$;
we obtain a free basis $\{ F_* b x_1^{\gamma_1} \cdots x_d^{\gamma_d} \,|\, b\in\mathcal{B}, 0\leq \gamma_1, \dots, \gamma_d < p \}$
for the free $R$-module $F_* R$. As an $R$-module, $\Hom_R(F_* R, R)$ is generated by the projections
$\pi_{b,\gamma_1, \dots, \gamma_d}$ onto the
free summands $F_* b x_1^{\gamma_1} \cdots x_d^{\gamma_d}$; moreover,
$\pi_{b,\gamma_1, \dots, \gamma_d}=\pi \circ F_* b^{-1} x_1^{p-\gamma_1} \cdots x_d^{p-\gamma_d}$
hence every element in $\Hom_R(F_* R, R)$ is given by $\pi \circ F_* u$ for some $u\in R$.

It is enough to establish (b) when $V$ is generated by one element $v\in F_* R^\alpha$. The equality
$\Phi(F_* R  U v)=I_1 (U R v)$
in (b) now follows from the fact that both sides are obtained as the outcome of the same calculation.
\end{proof}

Given an $n\times n$ matrix $U$ with entries in $R$, Proposition 7.3 exhibits (in the F-finite case) a correspondence between submodules $E(V)$ of $E^n$ fixed under the Frobenius map $U^t T$ and submodules $V$ of $R^n$ for which the near-splitting $\Phi \circ F_* U$  composed with the quotient map $R^n \rightarrow R^n/V$ does not vanish. In addition, $R^n /V$ has the same annihilator as its Matlis dual $E(V)$.

Thus, if one is given a the near-splitting $\Phi \circ F_* U$, finding all prime annihilators of submodules $V$ of $R^n$ for which the near-splitting $\Phi \circ F_* U$  composed with the quotient map $R^n \rightarrow R^n/V$ does not vanish, is equivalent to finding all annihilators of submodules $E(V)$ of $E^n$ fixed under the Frobenius map $U^t T$, and this we can do by applying the algorithm described in section 6, provided that the resulting Frobenius map is not nilpotent.

We conclude by re-interpreting the examples from section \ref{Section: The algorithm in action and two calculations}
in the context of Frobenius near-splittings.

\begin{ex}
Let
$R=\mathbb{Z}/2\mathbb{Z}[x,y,z]$, let
$$U=\left( \begin{array}{cc} x^3+y^3+z^3 & xy^2z^5 \\ x(y^2 + z^2) & x^3 \end{array} \right)$$
and consider the Frobenius near-splitting $\phi=(\pi\oplus\pi) \circ U$.
In the first example of section \ref{Section: The algorithm in action and two calculations}
we found special primes $P_0=0$, $P_1=(x, y+z) R$ and $P_2=(x,y,z)R$. These now give us the following three submodules
$V_1=0$,
$$V_2=\left(P_1{R^2}\right)^{\star U}=
\Image
\left( \begin{array}{ccccc}
x & y & z & 0 & 0\\
0 & 0 & 0 & x & y+z\\
\end{array}
\right)
$$
and
$V_3=\left((x,y,z) R^2\right)^{\star U}=(x,y,z) {R^2}$
of $F_* R^2$ which are
compatible with $\phi$
and such that $R^2/V_i$ has annihilator $P_i$ for $i=0,1,2$.
\end{ex}

\begin{ex}
Let $R=\mathbb{Z}/2\mathbb{Z}[x,y,z]$, $f=x^3+y^3+z^3$, $g=x^2+z^4$, define
$$U=\left( \begin{array}{cc} x f & y f \\ x g & y g \end{array} \right)$$
and consider the Frobenius near-splitting $\phi=(\pi\oplus\pi) \circ U$.
In the second example of section \ref{Section: The algorithm in action and two calculations}
we found special primes $P_0=0$, $P_1=(x,z)$ and $P_2=(x,y,z)$.
These now give us the following three submodules
$V_1=0$,
$$V_2=\left(P_1{R^2}\right)^{\star U}=
\Image
\left( \begin{array}{ccccc}
x & y & z & 0 & 0\\
0 & 0 & 0 & x & z\\
\end{array}
\right)
$$
and
$V_3=\left((x,y,z) R^2\right)^{\star U}=(x,y,z) {R^2}$
of $F_* R^2$ which are
compatible with $\phi$
and such that $R^2/V_i$ has annihilator $P_i$ for $i=0,1,2$.
\end{ex}

\end{document}